\documentclass [11pt]{amsart}
\usepackage{amsthm, amsmath, amssymb,latexsym}
\usepackage[usenames,dvipsnames]{color}
\usepackage[all]{xy}

\newtheorem{theorem}{Theorem}[section]
\newtheorem{lemma}[theorem]{Lemma}
\newtheorem{corollary}[theorem]{Corollary}
\newtheorem{proposition}[theorem]{Proposition}
\newtheorem{remark}[theorem]{Remark}

\newcommand{\nc}{\newcommand} 
\nc{\cH}{{\mathcal H}}
\nc{\cA}{{\mathcal A}}
\nc{\cG}{{\mathcal G}}
\nc{\cC}{{\mathcal C}}
\nc{\cO}{{\mathcal O}}
\nc{\cI}{{\mathcal I}}
\nc{\cB}{{\mathcal B}}
\nc{\cY}{{\mathcal Y}}
\nc{\cK}{{\mathcal K}} 
\nc{\cX}{{\mathcal X}}
\nc{\cS}{{\mathcal S}}
\nc{\cE}{{\mathcal E}}
\nc{\cF}{{\mathcal F}}
\nc{\cZ}{{\mathcal Z}}
\nc{\cQ}{{\mathcal Q}}
\nc{\cN}{{\mathcal N}}
\nc{\cP}{{\mathcal P}}
\nc{\cL}{{\mathcal L}}
\nc{\cM}{{\mathcal M}}
\nc{\cT}{{\mathcal T}}
\nc{\cW}{{\mathcal W}}
\nc{\cU}{{\mathcal U}}
\nc{\cJ}{{\mathcal J}}
\nc{\cV}{{\mathcal V}}
\nc{\bH}{{\mathbb H}}
\nc{\bA}{{\mathbb A}}
\nc{\bG}{{\mathbb G}}
\nc{\bC}{{\mathbb C}}
\nc{\bO}{{\mathbb O}}
\nc{\bI}{{\mathbb I}}
\nc{\bB}{{\mathbb B}}
\nc{\bY}{{\mathbb Y}}
\nc{\bK}{{\mathbb K}} 
\nc{\bX}{{\mathbb X}}
\nc{\bS}{{\mathbb S}}
\nc{\bE}{{\mathbb E}}
\nc{\bF}{{\mathbb F}}
\nc{\bZ}{{\mathbb Z}}
\nc{\bQ}{{\mathbb Q}}
\nc{\bN}{{\mathbb N}}
\nc{\bP}{{\mathbb P}}
\nc{\bL}{{\mathbb L}}
\nc{\bM}{{\mathbb M}}
\nc{\bT}{{\mathbb T}}
\nc{\bW}{{\mathbb W}}
\nc{\bU}{{\mathbb U}}
\nc{\bD}{{\mathbb D}}
\nc{\bJ}{{\mathbb J}}
\nc{\bV}{{\mathbb V}}
\nc{\bbZ}{{\mathbb Z}}
\nc{\bR}{{\mathbb R}}
\nc{\fr}{{\rightarrow}}
\nc{\co}{{\nabla}}

\newcommand{\la}{\longrightarrow}
\nc{\cu}{{\barline{\nabla}}}

\newcommand{\ra}{{\rightarrow}}

\begin{document}
\author{Elisabetta Colombo, Paola Frediani}
\title {On the dimension of  totally geodesic
submanifolds in the Prym loci} 
\dedicatory{Dedicated to Fabrizio Catanese on the occasion of his 70th birthday}

\address{Universit\`{a} di Milano} \email{elisabetta.colombo@unimi.it}
\address{Universit\`{a} di Pavia} \email{paola.frediani@unipv.it}

\thanks{The authors are members of GNSAGA of INdAM.
The authors were partially supported by national MIUR funds,
PRIN  2017 Moduli and Lie theory,  and by MIUR: Dipartimenti di Eccellenza Program
   (2018-2022) - Dept. of Math. Univ. of Pavia.
    } 

\begin{abstract}

In this paper we give a bound on the  dimension of a totally geodesic submanifold of the moduli space of polarised abelian varieties of a given dimension, which is contained in the Prym locus of a (possibly) ramified double cover. This improves the already known bounds. The idea  is to adapt the techniques introduced by the authors in collaboration with A. Ghigi and G. P. Pirola for the Torelli map to the case of the Prym maps of (ramified) double covers. 
\end{abstract}


\maketitle




\section{Introduction}
The purpose of this paper is to improve the estimates obtained in \cite{cfprym}, \cite{cframified} on the maximal dimension of a germ of a totally geodesic submanifold of the moduli space of polarised abelian varieties of a given dimension,  which is contained in the Prym locus of a (possibly) ramified double cover. The idea is to adapt to the Prym case the technique developed in \cite{cfg} and \cite{fp} to give a bound on the maximal dimension of a germ of a totally geodesic submanifold contained in the Torelli locus (see also \cite{gpt}).

Denote by $ {\mathcal R}_{g,r}$ the moduli space of isomorphism classes of triples $[(C, \alpha, R)]$ where $C$ is a smooth complex projective curve of genus $g$, $R$ is a reduced effective divisor of degree $2r$ on $C$ and $\alpha$ is a line bundle on $C$ such that $\alpha^2={\mathcal O}_{C}(R)$. 

A point $[(C, \alpha, R)] \in {\mathcal R}_{g,r}$ determines a double cover of $C$, $\pi:\tilde{C}\rightarrow C$ branched on $R$, with $\tilde{C}=Spec({\mathcal O}_{C}\oplus \alpha^{-1})$.

This defines the Prym variety $P(C,\alpha,R)$ which is the connected component containing 0 of the kernel of the norm map $Nm_{\pi}: J \tilde{C} \ra JC$.
If $r>0$, $\ker Nm_{\pi}$ is connected. The Prym variety $P(C,\alpha,R)$ is an abelian variety of
 dimension $g-1+r$ endowed with a  polarization $\Xi$, which is   induced by restriction of the principal polarisation on $J\tilde{C}$. The polarisation $\Xi$  is of type
 $\delta=(1,\ldots, 1, \underbrace{2, ...,2}_{g \textnormal{ times }} \ )$
for $r>0$, while
if $r=0,1$ it is twice a principal polarisation and we endow $P(\tilde{C}, C)$ with this principal polarisation. 
  Denote by $ {\mathcal A}^{\delta}_{g-1+r}$ is the moduli space of 
 abelian varieties of dimension $g-1+r $ with a polarization of type $\delta$. 
 The Prym map  $P_{g,r}: {\mathcal R}_{g,r} \rightarrow {\mathcal A}^{\delta}_{g-1+r}$ is defined as follows: 
  $ P_{g,r}([C, \alpha, R]):= [(P(C,\alpha,R), \Xi)]$. 
 
The map $P_{g,r}$ is generically finite, if and only if $\dim {\mathcal R}_{g,r} \leq \dim {\mathcal A}^{\delta}_{g-1+r},$
and this holds if: either $r\geq 3$ and $g\geq 1$, or $r=2$ and $g\geq 3$, $r=1$ and $g\geq 5$,  $r=0$ and $g \geq 6$.
If $r =0$  the Prym map is generically injective for $g \geq 7$  (\cite{fs}, \cite{ka}). 
If $r>0$, the works of Marcucci and Pirola \cite{mp},  Marcucci and Naranjo \cite{mn} and Naranjo and Ortega \cite{no}  show the generic injectivity in all the cases except for $r=2$, $g =3$, which was previously studied in \cite{nr}, \cite{bcv} and for which the degree of the Prym map is $3$. 
Recently a global Prym-Torelli theorem was proved for all $g$ and $r\geq 6$ (\cite{ikeda} for $g=1$, \cite{no1} for all $g$). 

In \cite{cfgp} an analogous question to the Coleman Oort conjecture for the  Prym maps $P_{g,r}$ for $r=0,1$ was formulated. Namely, the authors asked whether there exist Shimura subvarieties (hence totally geodesic subvarieties) of ${\mathcal A}^{\delta}_{g-1+r}$ ($r=0,1$) generically contained in the Prym loci. 

In \cite{cfgp} and in  \cite{fg} examples of Shimura curves generically contained in the (possibly ramified) Prym loci have been constructed for low values of $g$ and the computations of such examples suggest that they should not exist if $g$ is sufficiently high. 

Assume that $\dim {\mathcal R}_{g,r} \leq \dim {\mathcal A}^{\delta}_{g-1+r}$, so that the differential of the Prym map is generically injective and denote by $ {\mathcal R}^0_{g,r}$ the open subset of $ {\mathcal R}_{g,r}$ where the differential of the Prym map is injective. 
In this paper we give an estimate on the maximal dimension of a germ of a totally geodesic submanifold generically contained in the Prym loci, passing through a point $[(C, \alpha, R)] \in {\mathcal R}^0_{g,r}$ in terms of the gonality $k$ of the curve $C$. The  results are summarised  the following Theorem and they improve the estimates given in  \cite[Theorem 3.2]{cfprym} and \cite[Theorem 3.2]{cframified}.

\begin{theorem}
\label{bau-bau}
(Theorems \ref{gonality}, \ref{gonality-general}).
Let $[C, \alpha, R] \in {\mathcal R}^0_{g,r}$, where $C$ is a curve of genus $g>0$. Denote by $k$ its gonality and assume that $C$ has no involutions and that $g+r \geq k+3$. Denote by $Y$ a totally geodesic subvariety  contained in $P_{g,r}({\mathcal R}^0_{g,r})$ and passing through $P_{g,r}([C, \alpha, R])$. 
\begin{enumerate}
\item If  $ r >k+1$,  then $$\dim(Y) \leq \frac {3}{2}g - \frac{1}{2}  + r +k .$$
\item If $g +r \geq k+4$, $ r \leq k+1$, $b+r \geq 5$, then 
$$\dim(Y) \leq \frac {3}{2}g - \frac{1}{2} + \frac{b}{2} + r +k - \frac{h^0(F \otimes \alpha^{-1})}{2} \leq \frac {3}{2}g +2+  r +k. $$
\item If $g+r \geq k+4+m$, $r \leq k+1$, $b+r=5-m$, $m \geq 1$, then 
$$\dim(Y) \leq\frac {3}{2}g  +  k + 2 - \frac{m}{2} + \frac{r}{2}.$$

\end{enumerate}

\end{theorem}

From this, using that the gonality $k \leq [\frac{g+3}{2}]$,  we deduce the following result which improves the estimates obtained in \cite[Theorem 3.4]{cfprym} , \cite[Theorem 3.4]{cframified}.

\begin{theorem}
\label{mao-mao}
 (Corollaries \ref{stima2}, \ref{stima-general-final}). 
  Let $Y$ be a germ of a totally geodesic submanifold of
  ${\mathcal A}^{\delta}_{g-1+r}$ which is contained in $P_{g,r}( {\mathcal R}^0_{g,r})$, with $g \geq 3$. Assume that there exists a point $[C, \alpha, R] \in Y$  such that $C$ has no involutions. 
  \begin{enumerate}
  \item If $ g<2r-5$, then $\dim Y\leq 2g+r$ if $g$ even,  $\dim Y\leq 2g+r+1$ if $g$ is odd.
  \item If $g \geq 2r-5$, $r \geq 5$,  then $\dim Y\leq 2g+r+3 $.
 \item If $r=4$, if $g \geq 4$, then $\dim Y \leq 2g+4$ if $g$ is even, $\dim Y \leq 2g+5$ if $g$ is odd. 
\item If $r=3$, if $g \geq 8$, then $\dim Y \leq 2g+3$ if $g$ is even, $\dim Y \leq 2g+4$ if $g$ is odd. 
\item  If $r=2$, if $g \geq 12$, then $\dim Y \leq 2g+2$ if $g$ is even, $\dim Y \leq 2g+3$ if $g$ is odd. 
\item If $r=1$, if $g \geq 16$, then $\dim Y \leq 2g+1$ if $g$ is even, $\dim Y \leq 2g+2$ if $g$ is odd. 
\item If $r=0$, if $g \geq 20$, then $\dim Y \leq 2g$ if $g$ is even, $\dim Y \leq 2g+1$ if $g$ is odd. 

  \end{enumerate}
\end{theorem}

The improvements in the above estimates are obtained using ideas introduced in \cite{fp} that allow to use more than one quadric to compute the second fundamental form of the Torelli map. Here we need to adapt this technique to the more complicated case of the Prym maps.

 In the cases in which the assumptions of this theorem do not hold (i.e. for low values of $g$ and $r$), we still have the estimates done using only one quadric (Theorem   \ref{stima8}), as in \cite{cfprym} and \cite{cframified}, that are improved here thanks to a more careful study of the base locus of the linear system  $M:=K_C\otimes \alpha  \otimes F^{-1}$, where $F$ is the $g^1_k$ on the curve $C$. 

The result is the following
 \begin{theorem}
 (Theorem  \ref{stima8}). 
  Let $Y$ be a germ of a totally geodesic submanifold of
  ${\mathcal A}^{\delta}_{g-1+r}$ which is contained in $P_{g,r}( {\mathcal R}^0_{g,r})$.  
 
     \begin{enumerate}
   \item If $ g=1$, $r =4$,  then $\dim Y\leq 7$.
   \item If $ g=2$, $r =4$,  then $\dim Y\leq 9$.
 \item If $ g=3$, $r =4$,  then $\dim Y\leq 12$.
\item If $ 2 \leq g\leq 7$, $r=3$,  $\dim Y\leq \frac{9}{4} g +4$, if $g$ is even,  $\dim Y\leq \frac{9}{4} g +\frac{17}{4}$, if $g$ is odd.  
\item If $4 \leq g \leq 11$, $r =2$,  then $\dim Y\leq \frac{9}{4}g+4$, if $g$ is even, $\dim Y\leq \frac{9}{4} g +\frac{17}{4}$, if $g$ is odd. 
\item If $6 \leq g \leq 15$, $r =1$,  then $\dim Y\leq \frac{9}{4}g+\frac{5}{2}$, if $g$ is even, $\dim Y\leq \frac{9}{4} g +\frac{11}{4}$, if $g$ is odd. 
\item If $8 \leq g \leq 19$, $r =0$,  then $\dim Y\leq \frac{9}{4}g+1$, if $g$ is even, $\dim Y\leq \frac{9}{4} g +\frac{5}{4}$, if $g$ is odd.

  \end{enumerate}
\end{theorem}

Finally, in the case $r \geq 4$,  in Proposition \ref{k2r4} we get a better estimate than the one in Theorem \ref{bau-bau}, if the curve $C$ is hyperelliptic, which is possible since a global Prym-Torelli theorem is proved in \cite{no1}, \cite{ikeda}.

The structure of the paper is the following:
In Section 2 we introduce the notation and recall the results on the second fundamental form of the  ramified Prym map obtained in \cite{cframified}. 

In Section 3 we do a careful study of the linear system $M:=K_C\otimes \alpha  \otimes F^{-1}$, where $F$ is the $g^1_k$ on the curve $C$. This technical part is crucial to get the estimates. 

In Section 4 we explain the technique introduced in \cite{cfg} and in \cite{fp}  and adapt it in the case of the Prym-canonical linear system $K_C \otimes \alpha$, to construct certain quadrics containing the Prym-canonical image of the curve $C$, where we are able to compute the second fundamental form of the Prym map. 

In Section 5 we prove Theorems \ref{bau-bau}, \ref{mao-mao} for $r \geq 5$. 

In Section 6 we prove Theorems \ref{bau-bau}, \ref{mao-mao} for $r \leq 4$. 

In Section 7 we do the estimate using only one quadric in the missing cases, that is when $r \leq 4$ and for low values of $g$. Finally we prove Proposition \ref{k2r4}, in which we allow $C$ to be hyperelliptic if $r \geq 4$.

\section{Second fundamental form of the Prym map}

In this section we recall  the results contained in  \cite[Section 2]{cframified}. Denote by  $C$ a smooth complex projective curve of genus $g$, let $R$ be a reduced divisor of degree $2r$ on $C$ and $\alpha$ a line bundle on $C$ such that $\alpha^2={\mathcal O}_{C}(R)$. To such data we can associate  a double cover $\pi:\tilde{C}\rightarrow C$ branched on $R$. The Prym variety $P(C,\alpha, R)$  associated with this data is the polarised abelian variety given by the connected component containing the origin of kernel of the norm map $Nm_{\pi}:J\tilde{C}\rightarrow JC$. For $r>0$ the kernel of the norm map is connected. For $r>1$ the polarisation  $\Xi$ is the restriction to $P(C,\alpha, R)$ of the principal polarisation $\Theta_{\tilde{C}}$ on the jacobian of $\tilde{C}$. For $r=0,1$ the polarisation $\Theta_{\tilde{C}|P(C,\alpha,R)}$ is twice a principal polarisation $\Xi$ and we endow $P(C,\alpha, R)$ with this principal polarisation. 

The Prym map $P_{g,r}: {\mathcal R}_{g,r} \rightarrow {\mathcal A}^{\delta}_{g-1+r},$ is the map that  associates to a point $[(C,\alpha,R)] \in {\mathcal R}_{g,r}$ the isomorphism class of its Prym variety $P(C,\alpha, R)$ with the polarisation $\Xi$.

The dual of the differential of the Prym map $P_{g,r}$ at a generic point $[(C, \alpha, R)]$  is given by the multiplication map
\begin{equation}
\label{dp}
(dP_{g,r})^* : S^2H^0(C, K_C \otimes \alpha) \ra H^0(C, K_C^2(R))
\end{equation}

and it is surjective at the generic point if either $r\geq 3$ and $g\geq 1$, or $r=2$ and $g\geq 3$, $r=1$ and $g\geq 5$,  $r=0$ and $g \geq 6$ (see \cite{lo}). 

In these cases we denote by ${\mathcal R}^0_{g,r}$ the non empty open subset of ${\mathcal R}_{g,r}$ where the Prym map $P_{g,r}$ is an immersion. 
Notice that if  $r \geq 3$, a global Torelli theorem is proved in \cite{no1} (and in \cite{ikeda} for $g=1$), hence  ${\mathcal R}^0_{g,r} =  {\mathcal R}_{g,r}$.

Let us consider the (orbifold) tangent bundle exact sequence of the Prym map
\begin{equation}
\label{tangent}
0 \rightarrow T_{{\mathcal R}^0_{g,r}} \rightarrow P_{g,r}^*T_{{\mathcal A}^{\delta}_{g-1+r}}  \rightarrow {\mathcal N_{{\mathcal R}^0_{g,r}/{\mathcal A}^{\delta}_{g-1+r}}} \rightarrow 0
\end{equation}

We endow ${\mathcal A}^{\delta}_{g-1+r}$ with the orbifold metric induced by the symmetric metric on the Siegel space ${\mathcal H}_{g-1+r}$. The dual of  the associated second fundamental form  with respect to the metric connection of the above exact sequence is a map
\begin{equation}
\rho_P: {\mathcal N^*_{{\mathcal R}^0_{g,r}/{\mathcal A}^{\delta}_{g-1+r}}} \rightarrow S^2 \Omega^1_{{\mathcal R}^0_{g,r}}.
\end{equation}

In \cite{cframified} a description of  this second fundamental form is given in terms of  the second fundamental form of the Torelli map of the covering curves $\tilde{C}$. 
Denote by $U$ the open subset of ${\mathcal R}^0_{g,r}$ where there is a universal family  ${\tilde{f}} :    \tilde{\mathcal {C}} \rightarrow U$ and the differential of the modular map $U \ra {\mathcal M}_{\tilde{g}}$, $[\tilde{C} \ra C] \mapsto [\tilde{C}]$ is injective.

At the point $b_0:= [(C,\alpha,R)] \in U$ corresponding to the $2:1$ cover $\pi:\tilde{C} \ra C$, the space $P_{g,r}^*\Omega^1_{{{\mathcal A}^{\delta}_{g-1+r}}, b_0}$ is isomorphic to $S^2H^0(K_C \otimes \alpha)$,    $\Omega^1_{{\mathcal R}^0_{g,r},b_0}$ is isomorphic to $ H^{0}(K_C^2(R))$, ${\mathcal N^*_{{{\mathcal R}^0_{g,r}/{\mathcal A}^{\delta}_{g-1+r, b_0}} }}\cong I_2(K_C \otimes \alpha)$, and the dual of the exact sequence \eqref{tangent} at the point $b_0$ becomes

$$0 \rightarrow I_2(K_C\otimes \alpha)  \rightarrow S^2H^0(K_C \otimes \alpha) \stackrel {m}\rightarrow H^0(K_C^2(R))
\rightarrow 0. $$

The dual of the second fundamental form of the Prym map at the point $b_0$  is a map

 \begin{equation}
 \label{II}
 \rho_P: I_2(K_C\otimes \alpha)\rightarrow S^2H^0(K_C^2(R))
\end{equation}

Denote by $I_2(K_{\tilde C})$ the kernel of the multiplication map $S^2H^0(K_{\tilde C}) \ra H^0(2K_{\tilde C})$ and by  $I_2(K_{\tilde C})^+ $ its invariant subspace by the action of the involution on $\tilde{C}$.  
Observe that we have the inclusion  $I_2(K_C\otimes \alpha)  \stackrel{\pi^*} \hookrightarrow  I_2(K_{\tilde{C}})^+$. Moreover by the projection formula we have $H^1(T_C(-R)) \cong H^1(T_{\tilde{C}})^+$. 
In \cite[Thm. 2.1, formula (2.13)]{cframified} we have shown that 
$\forall Q \in I_2(K_C\otimes \alpha) $, $\forall v_1, v_2 \in H^1(T_C(-R))$,  we have: 
\begin{equation}
\label{rho(Q)}
\rho_P(Q)(v_1  \odot v_2) = \tilde{\rho}(\pi^*Q)(v_1 \odot v_2).
\end{equation}

\section{Gonality}
Let $[C, \alpha, R] \in {\mathcal R}^0_{g,r}$, where $C$ is a curve of genus $g>0$. Denote by $k$ its gonality. We have $2\leq k\leq [\frac{g+3}{2}]$. Assume $g+r \geq k+3$. 
Let $F$ be  a line bundle  on $C$ of degree $k$  and such that $h^0(F)=2.$ Set $M:=K_C\otimes \alpha  \otimes F^{-1}$, where $K_C$ is the canonical bundle. Set $n+1:= h^0(M)$. 
From Riemann Roch we get 
\begin{equation}
\label{h0m-b}
h^0(M)=h^0(F \otimes \alpha^{-1})+2g-2+ r-k-(g-1)=h^0(F \otimes \alpha^{-1}) +g-1+r-k \geq g+r -(k+1)
\end{equation}
Hence $h^0(M) \geq 2$, since  $g+r \geq k+3$. 
The Clifford index  of $C$ is either $k-2$ (computed by $F$) or $k-3$. Denote by $B$ the base locus of $M$ and set $b:= deg(B)$.  
The line bundle $M$ has degree equal to $2g-2+r-k$. 
\begin{proposition}
Assume $g \geq 1$, $g +r \geq k+3$. 
\label{r>k+1}
\begin{enumerate}
\item If $ r >k+1$,  then $b=0$ and $h^0( F \otimes \alpha^{-1})=0$. 
\item If $r >k+2$, then $|M|$ is very ample.   
\item If $r = k+2$, and $D$ is an effective divisor of degree $s$ such that $h^0(M(-D)) = h^0(M) -1=n$, then
$s \leq 2$.  Moreover $s=2$ if and only if either $C$ is hyperelliptic and $F \cong  {\mathcal O}_C(D)$, or $|M|$ gives a birational map that contracts $D$. 

\end{enumerate}
\end{proposition}
\begin{proof}
Since $r > k+1$, then $ deg(M) =2g-2+r-k > 2g-1$, hence $B = \emptyset$. We have $deg(F \otimes \alpha^{-1}) = k-r<0$, so $h^0( F \otimes \alpha^{-1})=0$. 
If  $r > k+2$, then $ deg(M) =2g-2+r-k > 2g$, therefore $|M|$ is very ample. 

Assume $r = k+2$ and let $D$ be an effective divisor of degree $s$ such that $h^0(M(-D)) = h^0(M)-1$.

By Riemann Roch we have $h^0(F(D) \otimes \alpha^{-1})= h^0(F \otimes \alpha^{-1}) +  s -1 = s-1$, $\deg(F(D) \otimes \alpha^{-1}) = s-2$, so  Clifford's theorem implies $s \leq 2$ and $s=2$ if and only if $F(D) \otimes \alpha^{-1} \cong {\mathcal O}_C$. In the case $s=2$, either $C$ is hyperelliptic and $F \cong  {\mathcal O}_C(D)$, or $|M|$ gives a birational map to its image that contracts $D$. 
\end{proof}

\begin{corollary}
\label{cor1}
If  $g \geq 3$, $C$ is not hyperelliptic,  $g +r \geq k+3$, and $ r >k+1$, then the morphism given by $|M|$ is birational onto its image.
\end{corollary}
\begin{proof}

The proof follows from Proposition \ref{r>k+1}. 
\end{proof}

Assume now that $r \leq k+1$. We have 
\begin{equation}
\label{h1m-b}
h^1(M(-B)) = b + h^0(F \otimes \alpha^{-1}).
\end{equation} 

We have the following

\begin{lemma}
\label{lemma1}
Assume  $g+r \geq k+3$, $r \leq k+1$. If either $b \geq 2$, or $b =1$ and $h^0(F\otimes \alpha^{-1}) = 1$, then 
\begin{enumerate}
\item $b \leq 5-r -2h^0(F \otimes \alpha^{-1})$.
\item If $b\geq 2$, $r \leq 3 - 2h^0(F \otimes \alpha^{-1})$.
\item If $b=1=h^0(F \otimes \alpha^{-1})$, $r \leq 2$. 
\end{enumerate}
\end{lemma}
\begin{proof}
If either $b \geq 2$, or $b =1$ and $h^0(F\otimes \alpha^{-1}) = 1$, 
 \eqref{h0m-b} and \eqref{h1m-b} show that the linear system $|M(-B)|$ contributes to the Clifford index, hence 
 $$cliff(M(-B)) = 2g-2 +r-k-b -2(h^0(M(-B)) -1)= -r+ k-b+2-2h^0(F \otimes \alpha^{-1}) \geq k-3.$$
Therefore $b \leq 5-r -2h^0(F \otimes \alpha^{-1})$. 
Moreover $r \leq 5-b -2h^0(F \otimes \alpha^{-1}) \leq 3- 2h^0(F \otimes \alpha^{-1})$ if $b \geq 2$. If $b=1 = h^0(F \otimes \alpha^{-1})$, $r \leq 5-b -2h^0(F \otimes \alpha^{-1}) =2. $
\end{proof}

\begin{corollary}
\label{cor3}
Assume  $g+r \geq k+3$, $r \leq k+1$. Then we have
\begin{enumerate}
\item $b \leq 5$. 
\item  If $r \geq 4$, then $b \leq 1$. 
\item If $r =3$, then $b \leq 2$. 
\item If $r \geq 3$ and $b =1$, then $h^0(F \otimes \alpha^{-1}) =0$. 
 \end{enumerate}
\end{corollary}
\begin{proof}
The proof follows immediately from Lemma \eqref{lemma1}. 
\end{proof}

Let us now study the morphism given by $|M-B|$. Assume that $g+r \geq k+4$. 

\begin{proposition}
\label{b+r}
Assume  $g+r \geq k+ 4$, $r \leq k+1$. Let $D_1$ be an effective divisor of degree $s_1 \geq 3$ such that $h^0(M(-B-D_1)) = h^0(M(-B)) -1 = h^0(M)-1$. Then $s_1 \leq -r-b-2 h^0(F \otimes \alpha^{-1})+7$. Therefore $b+r \leq 4$. 
\end{proposition}

\begin{proof}
By Riemann Roch we have 
$$h^0(F(B+D_1) \otimes \alpha^{-1})= h^0(F \otimes \alpha^{-1}) + b + s_1 -1 \geq s_1-1 \geq 2,$$
$$h^1(F(B+D_1) \otimes \alpha^{-1}) = h^0(M(-B-D_1)) = h^0(M) -1 =$$
$$= h^0(F \otimes \alpha^{-1}) + g -2 -k+r\geq g+r - k -2 \geq 2,$$ 
since $g+r \geq k+4$.  Hence $F(B+D_1) \otimes \alpha^{-1}$ contributes to the Clifford index and we have 
$cliff(F(B+D_1) \otimes \alpha^{-1}) = k-r+b+s_1 - 2( h^0(F(B+D_1) \otimes \alpha^{-1}) -1) = k-r-b-s_1 -2 h^0(F \otimes \alpha^{-1}) +4 \geq k-3.$
So we have $s_1 \leq -r-b+7 -2 h^0(F \otimes \alpha^{-1}) $.  Thus $b+r \leq 7-s_1 -2 h^0(F \otimes \alpha^{-1})  \leq 4$. 
\end{proof}
\begin{corollary}
\label{cor4}
 If $g+r \geq k+4$, $r \leq k+1$, $b+r \geq 5$, then $s_1 \leq 2$, hence the morphism induced by $|M(-B)|$ is either birational on its image, or it has degree 2 on its image. 

\end{corollary}

\begin{proof}
By Proposition \ref{b+r} we have $s_1 \leq -(r +b) +7 -2h^0(F \otimes \alpha^{-1}) \leq 2$, so  the morphism induced by $|M(-B)|$ has degree at most $2$ on its image. 
\end{proof}

Assume now that $b+r \leq 4$. 

Consider an effective divisor $D_1$ as above of maximal degree $s_1 \geq 3$ with the property that $h^0(M(-B-D_1)) = h^0(M(-B)) -1 = h^0(M) -1 = n$.  

Then clearly the linear system $|M(-B-D_1)|$ is base point free and $h^0(M(-B-D_1)) = n \geq 2$, if $g+r \geq k+4$. Under this assumptions, in Proposition \ref{b+r}, we have shown that  $s_1 \leq 7 -(r+b) - 2h^0(F\otimes \alpha^{-1})$. 

We want to study the linear system $|M(-B-D_1)|$ and give sufficient conditions ensuring that the map associated with $|M(-B-D_1)|$ has at most degree 2 on its image.  Assume that $g+r \geq k+5$. 
We have the following 

\begin{proposition}
\label{b+r+s}
Assume $g+r \geq k+5$, $b+r \leq 4$, $r \leq k+1$. With the above notation,  assuming $s_1 \geq 3$, let $D_2$ be an effective divisor of degree $s_2>0$ such that $h^0(M(-B-D_1-D_2)) = h^0(M(-B-D_1)) -1 = n-1$. Then $ s_2\leq 9-(r+b+s_1) -2 h^0(F \otimes \alpha^{-1}) \leq 6-(r+b).$
\end{proposition}

\begin{proof}
 Assume that $D_2$ is an effective divisor of degree $s_2$ such that $h^0(M(-B-D_1-D_2))= h^0(M(-B-D_1)) -1 = n-1.$ Then we have 

\begin{gather*}
h^1(F(B+D_1+D_2) \otimes \alpha^{-1}) = h^0(M(-B-D_1-D_2)) =n-1=\\
=h^0(F \otimes \alpha^{-1}) + g-1 +r-k -2 \geq 2,
\end{gather*}
since $g+r \geq k+ 5$. 
\begin{gather*}
h^0(F(B+D_1+D_2) \otimes \alpha^{-1}) = h^0(M(-B-D_1-D_2)) + k+s_1+s_2+b-r -g+1 =\\
= h^0(F \otimes \alpha^{-1}) + s_1+s_2 +b -2 \geq s_2+1,
\end{gather*}
since $s_1 \geq 3. $ So $F(B+D_1+D_2) \otimes \alpha^{-1}$ contributes to the Clifford index, hence we have 
\begin{gather*}
cliff(F(B+D_1+D_2) \otimes \alpha^{-1}) = k + b + s_1 + s_2 - r - 2(h^0(F(B+D_1+D_2) \otimes \alpha^{-1}) -1) \\
= k-(r+b+s_1+s_2) + 6 -2 h^0(F \otimes \alpha^{-1}) \geq k-3,
\end{gather*}
if and only if 
$$s_2 \leq 9-(r+b+s_1) -2 h^0(F \otimes \alpha^{-1}) \leq 6-(r+b) -2 h^0(F \otimes \alpha^{-1}) \leq 6-(r+b).$$
\end{proof}


\begin{corollary}
\label{cor5'}
If $g+r \geq k+5$, $b +r = 4$, $r \leq k+1$, then the morphism induced by $|M(-B)|$ is either birational or it has degree 2 on its image. 
\end{corollary}
\begin{proof}
By Proposition \ref{b+r+s} we know that $s_2 \leq 2$, hence the morphism induced by $|M(-B-D_1)|$ is either birational, or it has degree 2 on its image, hence the same holds for the morphism induced by $|M(-B)|$. 

\end{proof}

Now we repeat the strategy explained above for the cases $b+r \leq 3$. Set $D_l$ an effective divisor of maximal degree $s_l$, $l=1,...,5$, such that $h^0(M(-B-D_1-...-D_l) )= n+1-l  \geq 2$. This holds if $g+r \geq k+3 + l$. 

\begin{proposition}
\label{b+s_1+s_l}
Assume that $g+r \geq k+4 + m$, $b+r \leq 5-m$, $r \leq k+1$, $m \geq 1$.
Assume $s_1,...,s_{m} \geq 3$, let $D_{m+1}$ be an effective divisor of degree $s_{m+1}>0$ such that $h^0(M(-B-D_1-...-D_{m+1})) = h^0(M(-B-D_1-...-D_{m})) -1 = n-m$. Then $ s_{m+1}\leq 5+2(m+1)-(r+b+s_1+...+s_{m}) -2 h^0(F \otimes \alpha^{-1}) \leq 7-m-(r+b)-2 h^0(F \otimes \alpha^{-1}) .$
\end{proposition}

\begin{proof}
We have $deg(F(B+D_1+...+D_{m+1}) \otimes \alpha^{-1})= k+b +s_1+...+s_{m+1} -r$,
\begin{gather*}
h^1(F(B+D_1+...+D_{m+1}) \otimes \alpha^{-1}) = h^0(M(-B-D_1-...-D_{m+1})) = n-m,\\
=h^0(F \otimes \alpha^{-1}) + g-1 +r-k -(m+1) \geq 2,
\end{gather*}
since $g+r \geq k+ 4+m$,
\begin{gather*}
h^0(F(B+D_1+...D_{m+1}) \otimes \alpha^{-1}) = \\
=h^0(F \otimes \alpha^{-1}) + s_1+...+s_{m+1} +b -(m+1) \geq  2m \geq 2,
\end{gather*}
since $s_i \geq 3, $ $i=1,...,m$, $s_{m+1} \geq 1$, $m \geq 1$. So $F(B+D_1+..+D_{m+1}) \otimes \alpha^{-1}$ contributes to the Clifford index, hence we have 
\begin{gather*}
cliff(F(B+D_1+..+D_{m+1}) \otimes \alpha^{-1}) = \\
=k-(r+b+s_1+...+s_{m+1}) + 2(m+2) -2 h^0(F \otimes \alpha^{-1}) \geq k-3,
\end{gather*}
if and only if 
$$s_{m+1} \leq 5+2(m+1)-(r+b+s_1+...+s_{m}) -2 h^0(F \otimes \alpha^{-1}) \leq 7-m-(r+b) -2 h^0(F \otimes \alpha^{-1}) .$$
\end{proof}

\begin{corollary}
\label{general}
If $g+r \geq k+4+m$, $b +r = 5-m$, then the morphism induced by $|M(-B)|$ is either birational or it has degree 2 on its image. 
\end{corollary}
\begin{proof}
By Proposition \ref{b+s_1+s_l} we know that $s_{m+1} \leq 2$, hence the morphism induced by $|M(-B-D_1-...-D_{m})|$ is either birational, or it has degree 2 on its image, hence the same holds for the morphism induced by $|M(-B)|$. 

\end{proof}

\section{Quadrics and estimate on the second fundamental form}\label{se4}

In this section we explain how to construct certain quadrics contained in $I_2(K_C \otimes \alpha)$ on which we are able to compute the second fundamental form of the Prym map $\rho_P$. Here we adapt the techniques introduced in \cite{cfg} and  \cite{fp} for the Torelli map to the case of the Prym maps. 

\subsection{The quadrics and the second fundamental form}\label{se41}
With the above notation, assume that  the morphism $f$ induced by $|M(-B)|$ is  birational  on its image.  Set $d = deg(M(-B))= 2g-2+r-k-b$, $n+1 = h^0(M(-B)) = h^0(M)= h^0(F \otimes \alpha^{-1}) +g-1+r-k $, and  fix now and for all two indipendent  sections $x,y \in H^0(F).$ Take a section $t \in H^0(M(-B))$  and consider the associated divisor $D(t)\in |M(-B)|.$ Assume that $t$ is such that:
\begin{enumerate}
\item $D(t)\cap (Z\cup B)=\emptyset, $ where $Z$ is the ramification divisor of the map $\phi_{\langle x, y \rangle}: C \ra \bP^1$. 
\item $D(t)=p_1+\dots +p_d$, $p_i\neq p_j$ if $i\neq j.$
\item The points $p_i$ are in general linear position: for any group of or $n$ distinct points $p_{i_1},\dots,  p_{i_n},$
we have $H^0(M(-B-(p_{i_1}+\dots  +p_{i_n})))=\langle t \rangle$.
\end{enumerate}
The last condition follows for instance  from the uniform lemma of Castelnuovo (see e.g.  \cite[Ch.3]{acgh}) since $f:C\to \bP^n$ is birational onto its image.

Consider the exact sequence induced by $t$
$$0\to \cO_C\stackrel{t}\to M(-B)\to M(-B)_D\to 0.$$
We  get  $$M(-B)_D\cong \sum\bC_{p_i},$$
where the last isomorphism follows from the choice of local trivializations of $M(-B).$
Let $W\subset H^0(M(-B))$ be complementary to $t:$ $H^0(M(-B))=\langle t \rangle \oplus W$, so that $\dim W=n$. 
Consider the induced injection $j: W\to H^0(\oplus \bC_{p_i})=\bC^d.$ 
 We can rewrite the linear uniform condition.
For any $s\in W$, $s\neq0$,  then the vector 
$j(s)=(a_1,\dots, a_d)$
has at most $n-1$ coordinates that are zero.

Let $\tau \in H^0({\cO}_C(B))$, such that $D(\tau) = B$, then  $\forall s \in W$, $\tilde{s} = \tau s$ and $\tilde{t} = \tau t$,  are sections in $H^0(M)$. 
Let $I_2(K_C \otimes \alpha)\subset S^2 H^0(K_C\otimes \alpha))$  be the kernel of the multiplication map 
$$m: S^2 H^0(K_C\otimes \alpha) \ra H^0(K_C^{\otimes 2}(R)).$$
Consider the quadric 
$$Q_{s} = x\tilde{t} \odot y\tilde{s} - x\tilde{s} \odot y\tilde{t}$$
Clearly $Q_s \in I_2(K_C \otimes \alpha),$ $ \forall s \in W$. 
Denote by $\pi: \tilde {C} \ra C$ the $2:1$ cover given by $\alpha$, and by $\psi: C \ra \bP^1$ the morphism induced by $\langle s,t\rangle$. Then clearly $\{p_1,...,p_d\} = Z(t) = \psi^{-1}(1:0)$. Set $\{T_i, \sigma(T_i)\} = \pi^{-1}(p_i)$, where $\sigma$ denotes the involution of the double cover $\pi$. For all $i = 1,...,d$, set $v_i:= \xi_{T_i} + \xi_{\sigma(T_i)} \in H^1(T_{\tilde C})^+ \cong H^1(T_C(-R))$.

Denote by $V = \langle v_1,...,v_d \rangle \subset H^1(T_{\tilde C})^+$, and set

\begin{equation}
\label{pi}
\beta: W\to S^2V^\ast, \ \beta(s)(v \odot w) = \rho_P(Q_s)( v \odot w) =\tilde{ \rho}(\pi^*(Q_s))(v \odot w), \ \forall v,w \in V.
\end{equation}
where the last equality follows from \eqref{rho(Q)}. 

The following  result has been proved in  \cite[(3.2),(3.3)]{cframified}, with the variant that all the quadrics $Q_s$, $s \in  W$, are taken into account.
\begin{theorem}
\label{diago}
Let $x_1,\dots x_d$ be the basis of $V^\ast$ dual to the basis $\{v_1,...,v_d\}$ of $V$. Then $$\beta(s)= \lambda \sum_{i=1}^d a_ix_i^2$$ where $\lambda\neq 0$ is a constant independent of $s$,  
$j(s)=(a_1,\dots,a_d)$ where $ j:W\to \bC^d$ is the evaluation map.
The quadrics $\beta(s)$ are simultaneously diagonalized and  for any  $s\neq 0,$ 
${\rm{rank}} (\beta(s))\geq d-n+1.$
\end{theorem}
\begin{proof} 
By  \cite[(3.2)]{cframified}, for $i \neq j$ we have: 
$$ \pi^*(Q_s)(T_i, T_j) = Q_s(p_i, p_j) =0.$$

Hence, 
$$\rho_P(Q) (v_i \odot v_j) = \tilde{\rho}(\pi^*(Q_s))(v_i \odot v_j) = 0, \ \forall i \neq j,$$
$$ \rho_P(Q) (v_i \odot v_i)=\tilde{\rho}(\pi^*(Q_s))(v_i \odot v_i) = c \mu_2(Q_s)(p_i),$$
where $\mu_2: I_2(K_C \otimes \alpha) \ra H^0(4K_C(R))$ is the second Gaussian map of the bundle $K_C \otimes \alpha$ (see \cite[section 2]{cframified}). 
 A local computation gives $\mu_2(Q_s) = \mu_{1,F}(x \wedge y) \mu_{1, M}(\tilde{s} \wedge \tilde{t})$, where, for a line bundle $L$, $\mu_{1,L}: \Lambda^2 H^0(L) \ra H^0(K_C \otimes L^2)$ denotes the first Gaussian map of $L$ (see e.g. \cite[Lemma 2.2]{cf2}).  In local coordinates one computes $$ \mu_1(\tilde{s} \wedge \tilde{t}) = (\tau s)' (\tau t) -  (\tau s) (\tau t)' = (\tau' s + \tau s') (\tau t) - (\tau s)( \tau' t + \tau t') = \tau^2(s't-st').$$

Hence $\mu_1(\tilde{s} \wedge \tilde{t}) (p_i) = \tau^2(p_i) (-s(p_i)t'(p_i)) =k s(p_i) = ka_i,$
where $k$ is a  non zero constant, that does not depend on $s$, by the assumptions on the section $t$. So 
$\mu_2(Q_s)(p_i) = (\mu_1(x \wedge y) \mu_1(\tilde{s} \wedge \tilde{t}))(p_i) = \lambda a_i$, for a non zero constant $\lambda$, which is independent of $s$.

 \end{proof}
\subsection{The zero locus of the quadrics and the estimate}
For an element $z \in V$, write $z = \sum_{i=1}^d z_i v_i$ and denote by $[z]:=[z_1,...,z_d] \in  \bP^{d-1} \cong  \bP(V)$.
Consider the locus
$$Z=\{[z] \in \bP^{d-1}: \beta(s)(z \odot z)=\tilde{\rho}(\pi^*(Q_s))(z \odot z)=0,  \ \forall s\in W\}.$$ 
 First we have the following
\begin{lemma} Set $H=\{[z] \in \bP^{d-1}: z_i=0,\ i>n\},$ then $Z \cap H=\emptyset,$  therefore  $\dim Z=  d-n-1.$ 
 \end{lemma}
\begin{proof} 
Notice that by the uniform position (see e.g.  \cite[Ch.3]{acgh}), we know that for all $i \in \{1,...,n\}$, there exists exactly  a section $s_i\in W$   such that  $s_i(p_j) =0$, $\forall j \in \{1,...,n\}$, $j \neq i$, $s_i(p_i) \neq 0$, $s_i(p_k) \neq 0$, $\forall k >n$, hence by Theorem \ref{diago} we get 
\begin{equation}
\label{bau}
\beta(s_i)=  a_{i,i}x^2_i+\sum_{j>n} a_{i,j}x^2_j,
\end{equation}
  with $a_{i,i} \neq 0$, $a_{i,j} \neq 0$, $\forall j >n$. 
Take $[z] = [z_1,...,z_n,0,...,0]\in H$ such that $\beta(s)(z \odot z) =0$, $\forall s \in W$.
Set $Q_i:= Q_{s_i},$ then $\tilde{\rho}(\pi^*(Q_i))(z \odot z)= a_{i,i} z_i^2 =0$ $\forall i$ if and only if  $z_i=0$, $\forall i =1,...,n$, which is impossible, since $[z] \in H$.  Therefore we have  $H\cap Z=\emptyset$
and then $\dim Z \leq d-1-n.$ Notice that $Z=\{[v]\in \bP^{d-1}: \tilde{\rho}(\pi^*(Q_i))(v \odot v)=0, \ i=1,\dots,n\}$, so $\dim Z = d-1-n$.


\end{proof}

We need to estimate the dimension of a linear space $\Pi\subset Z.$ Denote by $T$ the linear subspace of $V$ corresponding to $\Pi$. 

Consider the map $h:V\to V$,  $\pi (x_1,\dots, x_d)= (0,\dots,0,x_{n},\dots, x_d)$ 
The restriction of $h$ to $T$ is injective since $\Pi\subset Z$. 
By formula \eqref{bau} we can see $\beta(s_n)$ as a quadric in $h(V)$. 

We have the inclusion $$h(T)\subset \{v \in \pi(V) \ | \ \beta(s_n)(v \odot v) = \tilde{\rho}(\pi^*(Q_{n}))(v \odot v)=0\}.$$ Since
$\beta(s_{n})$ has rank $d-n+1$, $\dim(h(V)) = d-n+1$ and $h: T \la V$ is injective,  we get:

\begin{proposition}
\label{bound}
 Assume that  the morphism induced by $|M(-B)|$ is birational on its image.  Set $d = deg(M(-B))= 2g-2+r-k-b$, $n+1 = h^0(M(-B)) = h^0(M)= h^0(F \otimes \alpha^{-1}) +g-1+r-k$. 
With the above notation, let $\Pi$ be a linear subspace contained in $Z,$ and let $T$ be the corresponding  subspace  of $V$.  Then 
$$\dim T \leq  \frac{d-n+1}{2} = \frac{g+1-b}{2} - \frac{h^0(F \otimes \alpha^{-1})}{2}.$$ 
\end{proposition}

\section{The  case $b+r \geq5$}

In this section we apply Proposition \ref{bound} in the case $b+r \geq 5$. We have the following 
\begin{theorem}
\label{gonality}
Let $[C, \alpha, R] \in {\mathcal R}^0_{g,r}$, where $C$ is a curve of genus $g>0$. Denote by $k$ its gonality and assume that $C$ has no involutions and that $g+r \geq k+3$. Denote by $Y$ a totally geodesic subvariety  contained in $P_{g,r}({\mathcal R}^0_{g,r})$ and passing through $P_{g,r}([C, \alpha, R])$. 
\begin{enumerate}
\item If  $ r >k+1$,  then $$\dim(Y) \leq \frac {3}{2}g - \frac{1}{2}  + r +k .$$
\item If $g +r \geq k+4$, $ r \leq k+1$, $b+r \geq 5$, then 
$$\dim(Y) \leq \frac {3}{2}g - \frac{1}{2} + \frac{b}{2} + r +k - \frac{h^0(F \otimes \alpha^{-1})}{2} \leq \frac {3}{2}g +2+  r +k. $$

\end{enumerate}

\end{theorem}
\begin{proof}
Under the above assumptions in both cases, by Corollaries \ref{cor1}, \ref{cor4}, the linear system $|M(-B)|$ induces a birational map on its image (in the first case $B = \emptyset$). 
 Let $S$ be the  tangent space  of $Y$ at at $P_{g,r}([C, \alpha, R])$. Let $V$ be as above,  then $T=S\cap V$ is a linear subspace where all
the quadrics vanish, so by Proposition  \ref{bound} we get:  $\dim (S\cap V)\leq \frac{g+1-b}{2}- \frac{h^0(F \otimes \alpha^{-1})}{2}$.

Then $\dim S+ \dim V \leq 3g-3 + 2r +\dim (S\cap V)$, hence 
$$\dim S\leq 3g-3 + 2r -(2g-k-2+r-b)+\frac{g+1-b}{2}- \frac{h^0(F \otimes \alpha^{-1})}{2}= $$
$$=\frac {3}{2}g - \frac{1}{2} + \frac{b}{2} + r +k - \frac{h^0(F \otimes \alpha^{-1})}{2}.$$ 
In case (1) we have  $b = 0 = h^0(F \otimes \alpha^{-1})$, by Proposition \ref{r>k+1}. 
 So we get $\dim Y = \dim S \leq \frac {3}{2}g - \frac{1}{2}  + r +k $. 
In case (2),  we  get $\dim Y = \dim S \leq \frac {3}{2}g - \frac{1}{2} + \frac{b}{2} + r +k - \frac{h^0(F \otimes \alpha^{-1})}{2} $. Then we conclude, since  $b \leq 5$, by Corollary \ref{cor3}.
\end{proof}

\begin{remark}
In the second case of Theorem \ref{gonality},  since $b+r \geq 5$, we can apply Lemma \ref{lemma1}  and we have the following cases: 
\begin{enumerate} 
\item If $b \geq 2$, then $b+r =5$ and $h^0(F \otimes \alpha^{-1}) =0$, then $\dim(Y) \leq \frac {3}{2}g - \frac{1}{2} + \frac{b}{2} + r +k - \frac{h^0(F \otimes \alpha^{-1})}{2} = \frac {3}{2}g +2+ \frac{r}{2} +k$. 
\item If $b=1$, then $h^0(F \otimes \alpha^{-1}) =0$ , then $\dim(Y) \leq \frac {3}{2}g - \frac{1}{2} + \frac{b}{2} + r +k - \frac{h^0(F \otimes \alpha^{-1})}{2} = \frac {3}{2}g + r +k.$
\item If $b=0$, then $\dim(Y) \leq \frac {3}{2}g - \frac{1}{2} + r +k - \frac{h^0(F \otimes \alpha^{-1})}{2} \leq  \frac {3}{2}g - \frac{1}{2} + r +k  . $
\end{enumerate}
\end{remark}

\begin{corollary}
  \label{stima2}
  Let $Y$ be a germ of a totally geodesic submanifold of
  ${\mathcal A}^{\delta}_{g-1+r}$ which is contained in $P_{g,r}( {\mathcal R}^0_{g,r})$, with $g \geq 3$. Assume that there exists a point $[C, \alpha, R] \in Y$ such that $C$  has no involutions.
  \begin{enumerate}
  \item If $ g<2r-5$, then $\dim Y\leq 2g+r$ if $g$ even,  $\dim Y\leq 2g+r+1$ if $g$ is odd.
  \item If $g \geq 2r-5$, $r \geq 5$,  then $\dim Y\leq 2g+r+3 $.

  \end{enumerate}
  
\end{corollary}
\begin{proof}

In case (1), by assumption we have $r > \frac{g+5}{2} = \frac{g+3}{2} +1 \geq [(g+3)/2]+1 \geq  k+1$. Since  $k \geq 3$, $r \geq  5$, hence  the result follows by Theorem \ref{gonality}, using  that $k\leq [(g+3)/2].$ 

In case (2), the condition $ r \geq 5$ implies $g + r \geq \frac{g+3}{2} + 4 \geq k+4$.  So if $r \leq k+1$, we use the estimate (2) in Theorem \ref{gonality}, and the inequality $k\leq [(g+3)/2]$ to conclude that  $\dim Y\leq 2g+r+3 $. If $r >k+1$, we use  the estimate (1) in Theorem \ref{gonality} and we get $\dim Y\leq 2g+r$ if $g$ is even, $\dim Y\leq 2g+r+1$ if $g$ is odd. 

\end{proof}

\section{The case $b+r \leq 4$}
Let us now consider the case $b+r \leq 4$. We have the following
\begin{theorem}
\label{gonality-general}
Let $[C, \alpha, R] \in {\mathcal R}^0_{g,r}$, where $C$ is a curve of genus $g>0$,  denote by $k$ its gonality. Assume that $C$ has no involutions. Denote by $Y$ a totally geodesic subvariety  contained in $P_{g,r}({\mathcal R}^0_{g,r})$ and passing through $P_{g,r}([C, \alpha, R])$. 
If $g+r \geq k+4+m$, $r \leq k+1$, $b+r=5-m$, $m \geq 1$, then 
$$\dim(Y) \leq\frac {3}{2}g  +  k + 2 - \frac{m}{2} + \frac{r}{2}.$$

\end{theorem}

\begin{proof} 

By Corollaries \ref{cor5'}, \ref{general}, the linear system $|M(-B)|$ is birational on its image. Let $S$ be the  tangent space  of $Y$ at at $P_{g,r}([C, \alpha, R])$. Let $V$ be as above,  then $T=S\cap V$ is a linear subspace where all
the quadrics vanish, so by Proposition  \ref{bound} we get:  
$$\dim S\leq 3g-3 + 2r -(2g-k-2+r-b)+\frac{g+1-b}{2}- \frac{h^0(F \otimes \alpha^{-1})}{2}= $$
$$=\frac {3}{2}g - \frac{1}{2} + \frac{b+r}{2} + \frac{r}{2} +k - \frac{h^0(F \otimes \alpha^{-1})}{2} = $$ 
$$=\frac {3}{2}g + 2 -\frac{m}{2} + \frac{r}{2} +k - \frac{h^0(F \otimes \alpha^{-1})}{2}.$$ 
\end{proof}

\begin{corollary}
  \label{stima-general}
  Let $Y$ be a germ of a totally geodesic submanifold of
  ${\mathcal A}^{\delta}_{g-1+r}$ which is contained in $P_{g,r}( {\mathcal R}^0_{g,0})$, with $g \geq 4m$,  $r=5-m$, $m \geq 1$. Assume that there exists a point $[C, \alpha, R] \in Y$ such that $C$  has no involutions.Then $\dim Y  \leq 2g -m + 5 $ if $g$ is even, $\dim Y  \leq 2g -m + 6 $ if $g$ is odd.

\end{corollary}
\begin{proof} 
By assumption we have $g+r = g+5-m \geq   [\frac{g+3}{2}] + 4+m \geq k+4+m$, $r \leq 4 \leq k+1$,  hence Theorem \ref{gonality-general} gives the estimate: 
$ \dim(Y) \leq\frac {3}{2}g  + k - \frac{m}{2}+ 2+ \frac{r}{2} = \frac {3}{2}g  + k  - m+ \frac{9}{2}$, and we conclude using $k \leq [\frac{g+3}{2}].$

\end{proof}

In conclusion we have the following estimates: 
\begin{corollary}
  \label{stima-general-final}
  Let $Y$ be a germ of a totally geodesic submanifold of
  ${\mathcal A}^{\delta}_{g-1+r}$ which is contained in $P_{g,r}( {\mathcal R}^0_{g,0})$, with $r \leq 4$. Assume that there exists a point $[C, \alpha, R] \in Y$ such that $C$ has no involutions.\begin{enumerate}
\item If $r=4$, if $g \geq 4$, then $\dim Y \leq 2g+4$ if $g$ is even, $\dim Y \leq 2g+5$ if $g$ is odd. 
\item If $r=3$, if $g \geq 8$, then $\dim Y \leq 2g+3$ if $g$ is even, $\dim Y \leq 2g+4$ if $g$ is odd. 
\item  If $r=2$, if $g \geq 12$, then $\dim Y \leq 2g+2$ if $g$ is even, $\dim Y \leq 2g+3$ if $g$ is odd. 
\item If $r=1$, if $g \geq 16$, then $\dim Y \leq 2g+1$ if $g$ is even, $\dim Y \leq 2g+2$ if $g$ is odd. 
\item If $r=0$, if $g \geq 20$, then $\dim Y \leq 2g$ if $g$ is even, $\dim Y \leq 2g+1$ if $g$ is odd. 
\end{enumerate}
\end{corollary}

\section{Estimate with one quadric}
In the cases in which the assumptions of corollaries \ref{stima2},  \ref{stima-general} do not hold, we still have the estimates obtained in \cite{cfprym} and \cite{cframified} using only one quadric, which can be improved by Lemma \ref{lemma1} and Corollary \ref{cor3} in the case $r \leq k+1$. 

In fact, assume $g+r \geq k+3$, so that $h^0(M(-B)) \geq 2$.  With the above notation, choose $s \in W$, $t \in H^0(M(-B))$ satisfying conditions $(1)$ and $(2)$ in section 4.1 and such that $s(p_i) \neq 0$, $\forall i = 1,...,d$. By Theorem \ref{diago} we know that 
$$\beta(s)= \tilde{ \rho}(\pi^*(Q_s))= \lambda \sum_{i=1}^d a_ix_i^2,$$
where $a_i = s(p_i) \neq 0$, $\forall i$, hence 
${\rm{rank}} (\beta(s)) ={\rm{rank}}(\tilde{ \rho}(\pi^*(Q_s)))\geq d = 2g-2+r-k-b.$

Then 

\begin{theorem}
  \label{stima1}
  Assume that $[(C,\alpha,R)] \in {\mathcal R}^0_{g,r}$ where $C$ is a $k$-gonal curve of genus $g>0$ with $g+r\geq k+3$. Let $Y$ be a germ of a totally geodesic submanifold of
  ${\mathcal A}^{\delta}_{g-1+r}$ which is contained in $P_{g,r}( {\mathcal R}^0_{g,r})$ and passes through
  $P(C,\alpha,R)$.    Then 
   \begin{enumerate}
  \item If $r > k+1$,  then $\dim Y\leq 2g-2+\frac{3}{2}r + \frac{k}{2}$.
  \item If $r \leq k+1$, then:
  \begin{itemize}
   \item If $r \geq 4$, then   $\dim Y\leq 2g-\frac{3}{2} + \frac{k}{2} + \frac{3}{2} r$.
  \item If $r=3,$ then $\dim Y\leq 2g+\frac{7}{2} + \frac{k}{2}$. 
  \item If $r \leq 2$, then $\dim Y\leq 2g+\frac{1}{2} + \frac{k}{2} + \frac{3}{2}r.$
  \end{itemize}
   \end{enumerate}
\end{theorem}
  \begin{proof} 
  
  Since ${\rm{rank}} (\beta(s)) ={\rm{rank}}(\tilde{ \rho}(\pi^*(Q_s)))\geq d = 2g-2+r-k-b,$ we have 
  $$\dim Y\leq 3g-3 +2r -\frac{d}{2} = 2g-2 + \frac{3}{2}r + \frac{k}{2} +  \frac{b}{2}.$$
  In the first case $b=0$, so we get the same estimate as in \cite{cframified}. 
  If $r \leq k+1$, we use Corollary \ref{cor3}  in the case $r\geq 4$, and in the case $r=3$, since we have respectively $b \leq 1$, $b \leq 2$. If $r \leq 2$, by Lemma \ref{lemma1} we know that $b \leq 5$.

   \end{proof}
  
So, in the cases in which the assumptions of corollaries \ref{stima2},  \ref{stima-general} do not hold, we have the following 

 \begin{theorem}
  \label{stima8}
  Let $Y$ be a germ of a totally geodesic submanifold of
  ${\mathcal A}^{\delta}_{g-1+r}$ which is contained in $P_{g,r}( {\mathcal R}^0_{g,r})$.  
 
     \begin{enumerate}
   \item If $ g=1$, $r =4$,  then $\dim Y\leq 7$.
   \item If $ g=2$, $r =4$,  then $\dim Y\leq 9$.
 \item If $ g=3$, $r =4$,  then $\dim Y\leq 12$.
\item If $ 2 \leq g\leq 7$, $r=3$,  $\dim Y\leq \frac{9}{4} g +4$, if $g$ is even,  $\dim Y\leq \frac{9}{4} g +\frac{17}{4}$, if $g$ is odd.  
\item If $4 \leq g \leq 11$, $r =2$,  then $\dim Y\leq \frac{9}{4}g+4$, if $g$ is even, $\dim Y\leq \frac{9}{4} g +\frac{17}{4}$, if $g$ is odd. 
\item If $6 \leq g \leq 15$, $r =1$,  then $\dim Y\leq \frac{9}{4}g+\frac{5}{2}$, if $g$ is even, $\dim Y\leq \frac{9}{4} g +\frac{11}{4}$, if $g$ is odd. 
\item If $8 \leq g \leq 19$, $r =0$,  then $\dim Y\leq \frac{9}{4}g+1$, if $g$ is even, $\dim Y\leq \frac{9}{4} g +\frac{5}{4}$, if $g$ is odd.

  \end{enumerate}
  
\end{theorem}
\begin{proof}
By assumption  $g + r \geq [\frac{g+3}{2}] + 3 \geq k +3$.   Moreover if $r \leq 3$, then $r  \leq k+1$, so we use the estimate (2) in Theorem \ref{stima1} (and the inequality $k\leq [\frac{g+3}{2}]$).  
For $r=4$, if $g \leq 2$, $r =4 >k+1 =3$, so we use the estimate (1) in Theorem \ref{stima1} (and the inequality $k\leq [\frac{g+3}{2}]$). If $r =4$ and $g =3$, if $k=2$ we use the estimate (1) in Theorem \ref{stima1} and we get $\dim Y \leq 11$. If $k =3$, we use the estimate (2) in Theorem \ref{stima1} and we obtain $\dim Y\leq 12$.

\end{proof}

Assume now that $r \geq 4$, hence ${\mathcal R}^0_{g,r} = {\mathcal R}_{g,r}$, since a global Prym-Torelli theorem holds for $r \geq 3$ (\cite{no1}, \cite{ikeda}).  Let $[C, \alpha, R] \in {\mathcal R}_{g,r}$ be a hyperelliptic curve, so $k=2$. 
We have the following 

\begin{proposition}
\label{k2r4}
Let $[C, \alpha, R] \in {\mathcal R}_{g,r}$, where $C$ is a hyperelliptic curve of genus $g$, $r \geq 4$.  Denote by $Y$ a totally geodesic subvariety  contained in $P_{g,r}({\mathcal R}_{g,r})$ and passing through $P_{g,r}([C, \alpha, R])$. 
\begin{enumerate}
\item If  $ r >4$,  then $\dim(Y) \leq \frac {3}{2}g + \frac{3}{2}  + r.$
\item If $r=4$ and either $\alpha \otimes F^{-1}$ is not effective, or  $\alpha \otimes F^{-1} \cong {\mathcal O}_C(D)$ with $D$  effective and $ {\mathcal O}_C(D) \not \cong F$, then  $\dim(Y) \leq \frac {3}{2}g + \frac{11}{2}.$
\end{enumerate}

\end{proposition}
\proof 

Since $r\geq 4 > k+1 =3$, by Proposition \ref{r>k+1},  we have $b=0$ and $h^0(F \otimes \alpha^{-1}) = 0$. 
Moreover, if $r > 4= k+2$,  $|M|$ is very ample, hence the proof of Theorem \ref{gonality}, applies, so  we have the same estimate as in Theorem \ref{gonality}, (1): 
$\dim(Y) \leq \frac {3}{2}g - \frac{1}{2}  + r +2 = \frac {3}{2}g + \frac{3}{2}  + r.$
 
If $r=4 =k+2$, by Proposition  \ref{r>k+1}, (3), we know that if $\alpha \otimes F^{-1}$ is not effective, then $|M|$ is very ample, hence we have again the estimate  in Theorem \ref{gonality}, (1), so 
$\dim(Y) \leq \frac {3}{2}g - \frac{1}{2}  + r +2 = \frac {3}{2}g + \frac{11}{2}.$

Assume $r=4$ and $\alpha \otimes F^{-1} \cong {\mathcal O}_C(D)$ with $D$  effective. Then if $ {\mathcal O}_C(D) \not \cong F$, by Proposition  \ref{r>k+1}, (3), we know that $|M|$ gives a birational map that contracts $D$, so once again we have the estimate in Theorem \ref{gonality}, (1): $\dim(Y) \leq \frac {3}{2}g - \frac{1}{2}  + 4+2 = \frac {3}{2}g + \frac{11}{2}.$

\qed



%


\begin{thebibliography}{10}
\bibitem{acgh} Arbarello,~E., Cornalba,~M., Griffiths,~P., Harris, ~J. {\em
Geometry of algebraic curves, Vol. I}, Grundlehren der
Mathematischen Wissenschaften, 267. Springer-Verlag, New York, 1985.


\bibitem{bcv} Bardelli,  ~F., Ciliberto,  ~C., Verra,  ~A., Curves of minimal genus on a general abelian variety,  {\em Compos. Math.} 96 (1995), no. 2, 115--147. 
\bibitem{cf2} E.~Colombo and P.~Frediani.  \newblock Some results on
  the second {G}aussian map for curves.  \newblock {\em Michigan
  Math. J.}, 58(3):745--758, 2009.

 
  \bibitem{cfprym} E.~Colombo and P.~Frediani.  \newblock  A bound on the dimension of a totally geodesic submanifold in the Prym locus. {\em Collectanea Mathematica}, 70, n.1, (2019), 51-57.
 
 \bibitem{cframified} E.~Colombo and P.~Frediani.  \newblock Second fundamental form of the Prym map in the ramified case. "Galois Covers, Grothendieck-Teichmueller Theory and Dessins d'Enfants - Interactions between Geometry, Topology, Number Theory and Algebra",  \newblock {\em Springer Proceedings in Mathematics $\&$ Statistics}, Vol.330, (2020). 
 
  \bibitem{cfg}
E.~Colombo, P.~Frediani, and A.~Ghigi.
\newblock On totally geodesic submanifolds in the {J}acobian locus.
\newblock {\em International Journal of Mathematics}, 26(01):1550005, 2015.

\bibitem{cfgp} Colombo, ~E., Frediani, ~P., Ghigi, ~A., Penegini ~M., Shimura curves in the Prym locus.
{\em Communications in Contemporary Mathematics}, Vol. 21, N0. 2 (2019) 1850009 (34 pages). 
\bibitem{fg} P.~Frediani, G.P.~Grosselli. \newblock Shimura curves in the Prym loci of ramified double covers.  arXiv:2007.09646.
  \bibitem{fp}
 P.~Frediani, G.P.~Pirola.
\newblock On the geometry of the second fundamental form of the Torelli map.  arXiv:1907.11407. To appear in {\em Proceedings of the AMS}. DOI: https://doi.org/10.1090/proc/15291.







\bibitem{fs}  Friedman, Robert; Smith, Roy,  The generic Torelli theorem for the Prym map. {\em Invent. Math.} 67 (1982), no. 3, 473-490.


\bibitem{gpt} A.~Ghigi, ~P. Pirola, ~S. Torelli.  \newblock Totally geodesic subvarieties in the moduli space of curves.  \newblock {\em arXiv:1902.06098}. To appear in {\em Communications in Contemporary Mathematics.} https://doi.org/10.1142/S0219199720500200. 

  




\bibitem{ikeda} A. Ikeda,  \newblock Global Prym-Torelli Theorem for double coverings of elliptic curves,  {\em Algebr. Geom.} 7 (2020), no. 5, 544-560.

\bibitem{ka} Kanev, V. I., A global Torelli theorem for Prym varieties at a general point. (Russian)
{\em Izv. Akad. Nauk SSSR Ser. Mat.} 46 (1982), no. 2, 244-268, 431.


\bibitem{lo}
Lange, ~H., Ortega,  ~A., Prym varieties of cyclic coverings, 
 {\em Geom. Dedicata} 150  (2011), 391--403.
 

\bibitem{mn}
V.~Marcucci, J.C.~Naranjo, Prym varieties of double coverings of elliptic curves,  {\em Int. Math. Res. Notices} 6 (2014), 1689-1698.

\bibitem{mp}
V.~Marcucci, G.P.~Pirola, Generic Torelli for Prym varieties of ramified coverings, {\em Compositio Mathematica} 148 (2012), 1147-1170.


\bibitem{nr}
D.S.~Nagaraj, S. Ramanan, Polarisations of type $(1,2,\dots,2)$ on abelian varieties, {\em Duke Math. J. }80 (1995), 157--194.  

\bibitem{no}
J.C.~Naranjo, A.~Ortega, Generic injectivity of the Prym map for double ramified coverings,  {\em Trans. Amer. Math. Soc.} 371 (2019), 3627-3646.

\bibitem{no1}
J.C.~Naranjo, A.~Ortega, Global Prym-Torelli for double coverings ramified in at least $6$ points,   arxiv:2005.11108. To appear in {\em Journal of Algebraic Geometry}. 




  
  








\end{thebibliography}
\end{document}